\documentclass[11pt]{amsart}
\usepackage{cases}

\theoremstyle{plain}
\newtheorem{theorem}                {Theorem}      [section]
\newtheorem*{theorem*}                {Theorem \ref{thm:appl}}
\newtheorem{proposition}  [theorem]  {Proposition}
\newtheorem{corollary}    [theorem]  {Corollary}
\newtheorem{lemma}        [theorem]  {Lemma}
\newtheorem{conjecture}   {Conjecture}

\theoremstyle{definition}

\newtheorem{remark}       [theorem]  {Remark}
\newtheorem{definition}   [theorem]  {Definition}

\newcommand{\la}{\langle}
\newcommand{\ra}{\rangle}

\DeclareMathOperator{\trace}{trace} 

\DeclareMathOperator{\Div}{div} 
 \DeclareMathOperator{\riem}{Riem}
\DeclareMathOperator{\ricci}{Ricci}

\DeclareMathOperator{\cst}{constant}
 \DeclareMathOperator{\Hess}{Hess}
 
\DeclareMathOperator{\grad}{grad}
\DeclareMathOperator{\Int}{int}
\numberwithin{equation}{section}

\begin{document}

\title[Bochner-Simons formulas]{Bochner-Simons formulas and the rigidity of biharmonic submanifolds}

\author{Dorel~Fetcu}
\author{Eric~Loubeau}
\author{Cezar~Oniciuc}

\address{Department of Mathematics and Informatics\\
Gh. Asachi Technical University of Iasi\\
Bd. Carol I, 11 \\
700506 Iasi, Romania} \email{dfetcu@math.tuiasi.ro}

\address{D{\'e}partement de Math{\'e}matiques \\
LMBA, UMR 6205 \\
Universit{\'e} de Bretagne Occidentale \\
6, avenue Victor Le Gorgeu \\
CS 93837, 29238 Brest Cedex 3, France}
\email{Eric.Loubeau@univ-brest.fr}

\address{Faculty of Mathematics\\ Al. I. Cuza University of Iasi\\
Bd. Carol I, 11 \\ 700506 Iasi, Romania} \email{oniciucc@uaic.ro}



\begin{abstract} We find some integral formulas of Simons and Bochner type and use them to study biharmonic and biconservative submanifolds in space forms. We obtain rigidity results that in the biharmonic case represent partial answers to two well-known conjectures on such submanifolds in spheres.
\end{abstract}

\maketitle

\section{Introduction}

The rich history of using tensorial formulas to understand the geometry of hypersurfaces in Riemannian manifolds goes back to the $1968$ seminal paper \cite{JS} by J. Simons, where, after finding the expression of the Laplacian of the squared norm of the second fundamental form of a minimal submanifold, which in the (simpler) case of minimal hypersurfaces $M^m$ in $\mathbb{S}^{m+1}$ is
$$
\frac{1}{2}\Delta|A|^2=-|\nabla A|^2-|A|^2(m-|A|^2),
$$
where $A$ is the shape operator, the author uses it to prove a very important rigidity result for compact minimal submanifolds of the Euclidean sphere. 

Simons' results were then generalized to the case of constant mean curvature (CMC) hypersurfaces in space forms by K. Nomizu and B. Smyth \cite{NS} in $1969$, and by J. Erbacher \cite{E} in $1971$ and B. Smyth \cite{BS} in $1973$ to the even more general case of submanifolds with parallel mean curvature vector field (PMC) in space forms. 

In $1977$, S.-Y. Cheng and S.-T. Yau \cite{CY} proved a general Simons type formula for Codazzi tensors, i.e., symmetric $(1,1)$-tensors $S$ on an $m$-dimensional Riemannian manifold $M$ and satisfying the classical Codazzi equation $(\nabla_XS)Y=(\nabla_YS)X$:
\begin{align}\label{eq:CY}
\frac{1}{2}\Delta|S|^2=-|\nabla S|^2-\sum_{i=1}^{m}\langle S,\Hess(\trace S)\rangle-\frac{1}{2}\sum_{i,j=1}^{m}R_{ijij}(\lambda_i-\lambda_j)^2,
\end{align}
where $\lambda_i$ are the eigenvalues of $S$ and $R_{ijkl}$ are the components of the Riemannian curvature of $M$. Using this equation with $S=A$ one reobtains Nomizu and Smyth's result as well as the Simons' one, after rewriting the last term. 

However, when $S$ fails to satisfy the Codazzi condition, Formula \eqref{eq:CY} ceases to work. In this case, a very useful tool proved to be a non-linear Bochner type formula in a paper from $1993$ by N. Mok, Y.-T. Siu, and S.-K. Yeung \cite{MSY}. We will return with more details about this formula in the fourth section of our paper where we will use a similar one to study the geometry of biharmonic and biconservative hypersurfaces in space forms, especially in the Euclidean sphere. Note that, for compact CMC hypersurfaces in space forms this formula again leads to the Nomizu-Smyth equation of \cite{NS}, while when working with biharmonic, or, more generally, biconservative surfaces in a Riemannian manifold, and a non-Codazzi tensor, one recovers a result in \cite{LO}.

A biharmonic map $\phi:M\rightarrow N$ between two Riemannian manifolds is a critical point of the bienergy functional
$$
E_2:C^{\infty}(M,N)\rightarrow\mathbb{R},\quad E_{2}(\phi)=\frac{1}{2}\int_{M}|\tau(\phi)|^{2}\ dv,
$$
where $M$ is compact and $\tau(\phi)=\trace\nabla d\phi$ is the tension field of $\phi$. The corresponding Euler-Lagrange equation, also known as the biharmonic equation, was obtained by G. Y. Jiang \cite{J} in $1986$:
\begin{equation}\label{eq:Jiang}
\tau_2(\phi)=-\Delta\tau(\phi)-\trace R^N(d\phi,\tau(\phi))d\phi=0,
\end{equation}
where $\tau_{2}(\phi)$ is the bitension field of $\phi$, $\Delta=-\trace(\nabla^{\phi})^2 =-\trace(\nabla^{\phi}\nabla^{\phi}-\nabla^{\phi}_{\nabla})$ is the rough Laplacian defined on sections of $\phi^{-1}(TN)$ and $R^N$ is the curvature tensor of $N$, given by $R^N(X,Y)Z=[\bar\nabla_X,\bar\nabla_Y]Z-\bar\nabla_{[X,Y]}Z$. 

Since any harmonic map is also biharmonic, the objective is to study biharmonic maps that are not harmonic. These maps are called proper-biharmonic. When $\phi:M\rightarrow N$ is biharmonic and also an isometric immersion, $M$ is called a biharmonic submanifold of $N$. 

Biharmonic maps were introduced in $1964$ by J. Eells and J. H. Sampson \cite{ES} as a generalization of harmonic maps and nowadays this topic represents a well-established and very dynamic research direction in modern Differential Geometry. In Euclidean spaces, B.-Y. Chen \cite{C} proposed an alternative definition of biharmonic submanifolds. Chen's definition coincides with the previous one when the ambient space is $\mathbb{E}^n$. He has also conjectured that in $\mathbb{E}^n$ there are no proper-biharmonic submanifolds.

When the ambient space has (constant) non-positive sectional curvature all known results have suggested a similar conjecture called the Generalized Chen Conjecture (see \cite{CMO1,M,GNU,Ou}).

A special attention has been paid to biharmonic submanifolds in spheres and articles like \cite{BMO,CMO1,CMO,JHC} led to the following two conjectures.

\begin{conjecture}[\cite{BMO}]
Proper-biharmonic submanifolds of $\mathbb{S}^{n}$ are CMC.
\end{conjecture}

\begin{conjecture}[\cite{BMO}]\label{conj2}
The only proper-biharmonic hypersurfaces of $\mathbb{S}^{m+1}$ are (open parts of) either hyperspheres $\mathbb{S}^m(1/\sqrt{2})$ or standard products of spheres $\mathbb{S}^{m_1}(1/\sqrt{2})\times\mathbb{S}^{m_2}(1/\sqrt{2})$, $m_1+m_2=m$, $m_1\neq m_2$.
\end{conjecture}

The second conjecture remains difficult to prove even if we assume that the hypersurface is also CMC and compact. It actually is a problem of broader interest as a CMC hypersurface $M^m$ in $\mathbb{S}^{m+1}$ is biharmonic if and only if the squared norm of its shape shape operator is constant and equal to $m$ (see \cite{BMO,O}). Therefore, CMC hypersurfaces with $|A|^2=m$ are biharmonic and their classification is a natural step after S.~S.~Chern, M.~do~Carmo, and S.~Kobayashi studied minimal hypersurfaces with $|A|^2=m$ in \cite{CdCK} (for more details see also \cite{AdC}).

The most recent result to support these two conjectures were obtained by S. Maeta and Y. Luo in \cite{LM} and by S. Maeta and Y.-L. Ou in \cite{MO}. In this last article the authors proved that any compact proper-biharmonic hypersurface with constant scalar curvature in the Euclidean sphere has constant mean curvature. However, they did not proved that such a hypersurface is necessarily one of Conjecture \ref{conj2}.

Let $\phi:M\rightarrow (N,h)$ be a fixed map, where $M$ is compact and $h$ is a Riemannian metric on $N$. Then think of $E_2$ as a functional on the set of all Riemannian metrics on $M$. Critical points of this new functional are Riemannian metrics characterized by the vanishing of the stress-energy tensor $S_2$. This tensor satisfies
$$
\Div S_2=-\langle\tau_2(\phi),d\phi\rangle.
$$
For a submanifold $M$ in $N$, if $\Div S_2=0$, then $M$ is called biconservative and it is characterized by the fact that the tangent part of its bitension field vanishes.

This paper deals mainly with Conjecture \ref{conj2} under additional geometric hypotheses, and is organized as follows. First present a general collection of known (and one new) results on biharmonic and biconservative submanifolds and the stress-energy tensor of the bienergy. In Section $3$, we compute the Laplacian of the squared norm of the tensor $S_2$ for any hypersurface in a real space form and deduce a rigidity result for compact biconservative hypersurfaces with constant scalar curvature. It turns out however that this situation is less rigid than the biharmonic case as we find more examples than those of Conjecture \ref{conj2}. In the fourth section, we obtain a new general integral formula for tensors, apply it to $S_2$, and show that compact biconservative submanifolds with parallel normalized mean curvature vector field (PNMC) in space space forms of dimension less than ten and non-negative sectional curvature must be CMC. As a consequence, in dimension less than ten, we obtain a similar result to Corollary \ref{cor3final} replacing the scalar-curvature condition with nowhere vanishing of the mean curvature. 

\noindent {\bf Conventions.} We work in the smooth category and assume manifolds to be connected and without boundary. On compact Riemannian manifolds we consider the canonical Riemannian measure.

\section{Preliminaries}

In this section we briefly recall basic results on biharmonic and biconservative submanifolds and a general formula for the Laplacian of the biharmonic stress-energy tensor. 

The stress-energy tensor associated to a variational problem, first described by D.~Hilbert in \cite{H}, is a symmetric $2$-covariant or $(1,1)$-tensor $S$ conservative, i.e., $\Div S=0$, at critical points.

To study harmonic maps, P.~Baird and J.~Eells \cite{BE} (cf. also \cite{S}) introduced the tensor
$$
S=\frac{1}{2}|d\phi|^2g-\phi^{\ast}h
$$
for maps $\phi:(M,g)\rightarrow (N,h)$ and showed that $S$ satisfies the equation
$$
\Div S=-\langle\tau(\phi),d\phi\rangle,
$$
hence $\Div S$ vanishes when $\phi$ is harmonic. When $\phi:M\to N$ is an arbitrary isometric immersion, $\tau(\phi)$ is normal and therefore $\Div S=0$.

The stress-energy tensor $S_2$ of the bienergy, introduced in \cite{J} and studied in \cite{CMOP, FNO, Fu,LMO,MOR,MOR2,N}, is
\begin{align*}
S_2(X,Y)=&\frac{1}{2}|\tau(\phi)|^2\langle X,Y\rangle+\langle d\phi,\nabla\tau(\phi)\rangle\langle X,Y\rangle\\&-\langle d\phi(X),\nabla_Y\tau(\phi)\rangle-\langle d\phi(Y),\nabla_X\tau(\phi)\rangle
\end{align*}
and satisfies
$$
\Div S_2=-\langle\tau_2(\phi),d\phi\rangle.
$$
If $\phi:M\to N$ is an isometric immersion, then $(\Div S_2)^{\sharp}=-\tau_2(\psi)^{\top}$ and, unlike the harmonic case, $\Div S_2$ does not necessarily vanish.

\begin{definition} A submanifold $\phi:M\rightarrow N$ of a Riemannian manifold $N$ is called biconservative if $\Div S_2=0$, i.e., $\tau_2(\phi)^{\top}=0$.
\end{definition}

For hypersurfaces of space forms the biharmonic stress-energy tensor is parallel whenever the shape operator is so.

\begin{proposition}
Let $\phi:M^m\rightarrow N^{m+1}(c)$ be a non-minimal hypersurface. Then $\nabla S_2=0$ if and only if $\nabla A=0$.
\end{proposition}

\begin{proof} First assume that $\nabla A=0$. It then easily follows that the mean curvature function $f=(1/m)\trace A$ is a non-zero constant. Let $H=f\eta=(1/m)\tau(\phi)$ be the mean curvature vector field of $M$, where $\eta$ is the unit normal vector field. Since for a hypersurface $S_2=-(m^2/2)f^2I+2mfA$, one obtains $\nabla S_2=0$. 

Assume now that $\nabla S_2=0$. Denote by $W$ the set of all points of $M$ where the number of distinct principal curvatures is locally constant. This subset is open and dense in $M$. On each connected component of $W$, which is also open in $M$, the principal curvatures are smooth functions and the shape operator $A$ is (locally) diagonalizable.

We will work on such a connected component $W_0$ of $W$ and prove that $f$ is constant on $W_0$. As $W$ is open and dense this property will then hold throughout $M$, and combined with $\nabla S_2=0$ yields $\nabla A=0$.

Assume that $\grad f$ does not vanish identically on $W_0$. Take a connected and open subset $U$ of $W_0$ where $\grad f\neq 0$ and $f\neq 0$ at each point in $U$. Consider an orthonormal frame field $\{E_i\}$ on $U$ such that $AE_i=\lambda_iE_i$ and, from $(\nabla S_2)(E_i,E_j)=(\nabla S_2)(E_j,E_i)$ and $\nabla A$ symmetric, we have
$$
-m^2f(E_if)E_j+2m(E_if)\lambda_jE_j=-m^2f(E_jf)E_i+2m(E_jf)\lambda_iE_i,\quad\forall i,j\in\{1,\ldots,m\}.
$$
For $i\neq j$, it follows that 
$$
(2\lambda_j-mf)E_if=0.
$$

From here, we get that
\begin{equation}\label{eq:intro2}
(\lambda_i-\lambda_j)(2\lambda_j-mf)E_if=0,\quad\forall i,j\in\{1,\ldots,m\}.
\end{equation}

Since $\grad f\neq 0$, we can assume that there exists $i_0\in\{1,\ldots,m\}$ such that $E_{i_0}f\neq 0$ at any point in $U$. From \eqref{eq:intro2}, one obtains, on $U$,
$$
2\lambda_{i_0}\lambda_j-m\lambda_{i_0}f-2\lambda_j^2+m\lambda_jf=0,\quad\forall j\in\{1,\ldots,m\}
$$
and, therefore,
\begin{equation}\label{eq:intro3}
(2-m)mf\lambda_{i_0}-2|A|^2+m^2f^2=0.
\end{equation}
The squared norms of $A$ and $S_2$ are related by
$$
16m^2f^2|A|^2=4|S_2|^2-m^4f^4(m-8),
$$
and Equation \eqref{eq:intro3} shows that
$$
(2-m)mf\lambda_{i_0}=\frac{4|S_2|^2-m^5f^4}{8m^2f^2}.	
$$

If $m>2$, the above equation can be re-written as
\begin{equation}\label{eq:intro4}
2mf\lambda_{i_0}=\frac{4|S_2|^2-m^5f^4}{4(2-m)m^2f^2}.
\end{equation}
Since $\nabla S_2=0$, we have that $|S_2|$ is a constant on $M$ and the eigenvalues of $S_2$ also are constant functions on $M$, i.e.,  
$$
-\frac{m^2}{2}f^2+2mf\lambda_i=c_i=\cst.
$$
It follows, using \eqref{eq:intro4}, that on $U$, we have
$$
-\frac{m^2}{2}f^2+\frac{4|S_2|^2-m^5f^4}{4(2-m)m^2f^2}=c_{i_0},
$$
which gives a polynomial equation in $f^2$ with constant coefficients forcing $f$ to be constant on $U$ and contradicting $E_{i_0}f\neq 0$ at any point of $U$.

If $m=2$, Equation \eqref{eq:intro3} gives $|A|^2=2f^2$, which leads to $\lambda_1=\lambda_2$ on $U$. Therefore, $U$ is umbilical in $N$ and $f$ is constant on $U$. As we have already seen, this is a contradiction and we are finished.
\end{proof}

\begin{remark}
The case when $m\neq 4$ had already been proved, by a different method, in \cite{LMO}.
\end{remark}

\begin{remark} 
Hypersurfaces of space forms with $\nabla A=0$ were studied in \cite{L,R}. They can only have one or two distinct principal curvatures and they must be constant. When admitting two distinct principal curvatures they are intrinsically isometric to the product of two space forms and, using either the Moore Lemma or the Fundamental Theorem of hypersurfaces in space forms, one obtains a complete classification.
\end{remark}

The basic characterization of hypersurfaces in space forms in terms of $S_2$ is given by the following proposition.

\begin{proposition}[\cite{LMO}]
Let $\phi:M^m\rightarrow N^{m+1}(c)$ be a hypersurface in a space form $N$ and $S_2$ its biharmonic stress-energy tensor. Then we have$:$
\begin{enumerate}
\item if $m\neq 4$, then $S_2=0$ if and only if $M$ is minimal$;$ 

\item if $m=4$, then $S_2=0$ if and only if $M$ is either minimal or umbilical$;$

\item $S_2=a\langle,\rangle$, with $a\neq 0$, if and only if $m\neq 4$ and $M$ is umbilical and non-minimal.
\end{enumerate}
\end{proposition}

Essential to further computations are the following properties of the shape operator $A$.

\begin{lemma}\label{lemmaA}
Let $\phi:M^m\rightarrow N^{m+1}(c)$ be a hypersurface in a space form with the shape operator $A$. Then 
\begin{enumerate}

\item $A$ is symmetric$;$

\item $\nabla A$ is symmetric$;$

\item $\langle (\nabla A)(\cdot,\cdot),\cdot\rangle$ is totally symmetric$;$

\item $\Div A=\trace\nabla A=m\grad f$$.$
\end{enumerate}
\end{lemma}

The next result computes the Laplacian of the biharmonic stress-energy tensor and will be used to derive a Simons type equation for hypersurfaces of space forms.

\begin{theorem}[\cite{LO}]\label{thm:LO} Let $\phi:M\rightarrow N$ a smooth map between two Riemannian manifolds. Then the (rough) Laplacian of $S_2$ is the symmetric $(0,2)$ tensor
\begin{align}\label{eq:general}
&(\Delta S_2)(X,Y) =\\\nonumber &\Big\{ \la \Delta \tau(\phi), \tau(\phi)\ra - 2 |\nabla \tau(\phi)|^2 -2 \sum \la R(X_i,X_j)d\phi(X_i), \nabla_{X_j}\tau(\phi)\ra \\\nonumber
& -2 \la d\phi(\ricci(\cdot)),\nabla_{(\cdot)}\tau(\phi)\ra -2\la \nabla d\phi , \nabla^{2} \tau(\phi) \ra +  \la d\phi , \nabla(\Delta \tau(\phi))\ra \\\nonumber
&-\la \nabla(\trace R^{N}(d\phi(\cdot),\tau(\phi)) d\phi(\cdot)),d\phi \ra - \la \trace R^{N}(d\phi(\cdot),\tau(\phi)) d\phi(\cdot),\tau(\phi)\ra  \Big\} \la X,Y\ra\\\nonumber
& + 2 \la \nabla_{X}\tau(\phi) , \nabla_{Y}\tau(\phi)\ra + \sum \la R(X_i,X)d\phi(X_i) , \nabla_{Y}\tau(\phi)\ra\\\nonumber & +\sum  \la R(X_i,Y)d\phi(X_i), \nabla_{X} \tau(\phi)\ra \\\nonumber
&+ \la d\phi(\ricci(X)), \nabla_{Y}\tau(\phi)\ra + \la d\phi(\ricci(Y)), \nabla_{X}\tau(\phi)\ra \\\nonumber
& + 2\sum  \la \nabla d\phi(X_i,X),(\nabla^2 \tau(\phi))(X_i,Y)\ra + 2\sum  \la \nabla d\phi(X_i,Y),(\nabla^2 \tau(\phi))(X_i,X)\ra \\\nonumber
& -\la d\phi(X),\nabla_{Y}(\Delta \tau(\phi))\ra - \la d\phi(Y),\nabla_{X}(\Delta \tau(\phi))\ra 
+ \sum \la d\phi(X), R(X_i , Y)\nabla_{X_i}\tau(\phi)\ra \\\nonumber
&+\sum  \la d\phi(Y), R(X_i , X)\nabla_{X_i}\tau(\phi)\ra 
\\\nonumber
&+ \sum \la d\phi(X), (\nabla R)(X_i , X_i ,Y, \tau(\phi)) + R(X_i ,Y)\nabla_{X_i}\tau(\phi)\ra\\\nonumber & +\sum  \la d\phi(Y), (\nabla R)(X_i , X_i ,X, \tau(\phi)) + R(X_i ,X)\nabla_{X_i}\tau(\phi)\ra \\\nonumber &
+ \la d\phi(X), \nabla_{\ricci(Y)}\tau(\phi)\ra + \la d\phi(Y), \nabla_{\ricci(X)}\tau(\phi)\ra ,
\end{align}
where $\{X_i\}$ is a local orthonormal frame field.
\end{theorem}

\begin{remark} In equation \eqref{eq:general} we have $(\nabla^2\tau(\phi))(X,Y)=\nabla_X\nabla_Y\tau(\phi)-\nabla_{\nabla_XY}\tau(\phi)$, while $R$ is the curvature tensor in $\phi^{-1}(TN)$ and
\begin{align*}
(\nabla R)(X,Y,Z,\sigma)&=(\nabla_{X}R)(Y,Z,\sigma)\\\nonumber&=\nabla_{X}R(Y,Z)\sigma - R(\nabla_{X}Y,Z)\sigma - R(Y,\nabla_{X}Z)\sigma - R(Y,Z)\nabla_X\sigma.
\end{align*}
\end{remark}

Recall that the decomposition in its normal and tangent part of the biharmonic equation $\tau_2(\phi)=0$ of a hypersurface $M^m$ in $N^{m+1}$ yields
$$
\Delta f+f|A|^2-f\ricci^N(\eta,\eta)=0
$$
and
$$
2A(\grad f)+mf\grad f-2f(\ricci^N(\eta))^{\top}=0,
$$
where $(\ricci^N(\eta))^{\top}$ is the tangent component of the Ricci curvature of $N$ in the direction of $\eta$. It is easy to see that while any CMC hypersurface $M$ in a space form $N^{m+1}(c)$ is biconservative, $M$ is proper-biharmonic if and only if $|A|^2=cm$, hence $c$ must be positive.

\section{A Simons type formula for hypersurfaces and applications}

In \cite{MO}, asking only that a proper-biharmonic hypersurface be compact and with constant scalar curvature and using the Weitzenb\"ock formula for the differential $df$ of the mean curvature function, the authors proved that $f$ has to be constant. Our approach is different as we work with tensors and try to find the best tensorial formula in order to give an answer to Conjecture \ref{conj2}.

The Laplacian of the squared norm of the biharmonic stress-energy tensor of an immersed hypersurface can be computed and put to use to prove some rigidity results.

\begin{proposition}\label{p:delta} Let $\phi:M^m\rightarrow N^{m+1}(c)$ be a hypersurface in a space form. We have
\begin{align}\label{eq:delta1}
\frac{1}{2}\Delta|S_2|^2=&-|\nabla S_2|^2+4cm^4f^4-4m^3f^3(\trace A^3)-4m^2f^2|A|^2(cm-|A|^2)\\\nonumber
&+m^4(m-16)f^2|\grad f|^2+4m^2|A|^2|\grad f|^2+2m^2|A|^2\Delta f^2\\\nonumber &+4m^2f\langle\grad s,\grad f\rangle-8m^2\Div(f\ricci(\grad f))\\\nonumber & +\frac{m^5}{8}\Delta f^4-4cm^2(m-1)\Delta f^2-10m^2f\langle\tau_2^{\top}(\phi),\grad f\rangle\\\nonumber &-4m^2f^2\Div\left(\tau_2^{\top}(\phi)\right)-2\left|\tau_2^{\top}(\phi)\right|^2+4mf\langle\nabla\tau_2^{\top}(\phi),A\rangle.
\end{align}
\end{proposition}

\begin{proof} This is just an application of Formula \eqref{eq:general} of $\Delta S_2$ for an immersed hypersurface $M^m$ in a space form $N(c)$. For the sake of simplicity, we consider a point $p\in M$ and a geodesic frame field around it, and compute all terms at $p$. First, since $\tau(\phi)=mH$, we have
$$
\langle\Delta\tau(\phi),\tau(\phi)\rangle=m^2\langle\Delta H,H\rangle
$$
and
\begin{align*}
-2|\nabla\tau(\phi)|^2&=-2m^2|\nabla H|^2=-2m^2\sum|\nabla_{X_i}H|^2\\&=-2m^2\sum|-fAX_i+(X_if)\eta|^2\\&=-2m^2f^2|A|^2-2m^2|\grad f|^2.
\end{align*}

Next, using the expression of the curvature of a space form
\begin{equation}\label{eq:curv}
R^N(X,Y)Z=c\{\langle Y,Z\rangle X-\langle X,Z\rangle Y\},
\end{equation}
one obtains
\begin{align*}
-2\sum\langle R(X_i,X_j)d\phi(X_i),\nabla_{X_j}\tau(\phi)\rangle =&-2\sum\langle R^N(d\phi(X_i),d\phi(X_j))d\phi(X_i),\\&-mfAX_j+m\nabla^{\perp}_{X_j}H\rangle\\=&2cmf(1-m)(\trace A)=2cm^2f^2-2cm^3f^2.
\end{align*}
In the same way, we get
$$
-2\langle d\phi(\ricci(\cdot)),\nabla_{(\cdot)}\tau(\phi)\rangle=2mf\langle \ricci,A\rangle
$$
and then, since $\ricci=c(m-1)I+mfA-A^2$,
$$
-2\langle d\phi(\ricci(\cdot)),\nabla_{(\cdot)}\tau(\phi)\rangle=2c(m-1)m^2f^2+2m^2f^2|A|^2-2mf(\trace A^3).
$$

Since in the case of immersions we have $(\nabla d\phi)(X_i,X_j)=B(X_i,X_j)$, a direct computation using the Weingarten equation shows that
$$
-2\langle\nabla d\phi,\nabla^2\tau(\phi)\rangle=2mf(\trace A^3)-2m\langle A,\Hess f\rangle.
$$
Furthermore, for any hypersurface, we have
\begin{align*}
\langle A,\Hess f\rangle&=\sum\langle AX_i,\nabla_{X_i}\grad f\rangle=\sum\langle X_i,A(\nabla_{X_i}\grad f)\rangle\\&=\sum\langle X_i,\nabla_{X_i}A(\grad f)-(\nabla_{X_i}A)(\grad f)\rangle\\&=\Div(A(\grad f))-m|\grad f|^2
\end{align*}
and, therefore,
$$
-2\langle\nabla d\phi,\nabla^2\tau(\phi)\rangle=2mf(\trace A^3)-2m\Div(A(\grad f))+2m^2|\grad f|^2.
$$

The next term in the formula of $\Delta S_2$ is
\begin{align*}
\langle d\phi,\nabla(\Delta\tau(\phi))\rangle&=m\langle d\phi,\nabla(\Delta H)\rangle=m\sum\langle d\phi(X_i),\nabla_{X_i}(\Delta H)\rangle\\&=m\sum\left\{X_i\langle d\phi(X_i),\Delta H\rangle-\langle(\nabla d\phi)(X_i,X_i),\Delta H\rangle\right\}\\&=-\Div\tau_2^{\top}(\phi)-m^2\langle H,\Delta H\rangle. 
\end{align*}

Again using equation \eqref{eq:curv}, one obtains
\begin{align*}
-\langle\nabla(\trace R^N(d\phi(\cdot),\tau(\phi))d\phi(\cdot)),d\phi\rangle&=cm^2\langle\nabla H,d\phi(\cdot)\rangle=cm^2\sum\langle-fAX_i,d\phi(X_i)\rangle\\&=-cm^3f^2
\end{align*}
and
$$
-\langle\trace R^N(d\phi(\cdot),\tau(\phi))d\phi(\cdot)),\tau(\phi)\rangle=cm^3f^2.
$$

The expressions of the following terms can be obtained by some direct computation and also using Lemma \ref{lemmaA}, in the same way as above,
$$
2\langle\nabla_X\tau(\phi),\nabla_Y\tau(\phi)\rangle=2m^2f^2\langle AX,AY\rangle+2m^2(Xf)(Yf),
$$
\begin{align*}
\sum\langle R(X_i,X)d\phi(X_i),\nabla_Y\tau(\phi)\rangle=&\sum\langle R(X_i,Y)d\phi(X_i),\nabla_X\tau(\phi)\rangle\\=&cm(m-1)f\langle AX,Y\rangle,
\end{align*}
$$
\langle d\phi(\ricci(X)),\nabla_Y\tau(\phi)\rangle=-mf\langle\ricci(X),AY\rangle,
$$
$$
2\sum\langle\nabla d\phi(X_i,X),(\nabla^2\tau(\phi))(X_i,Y)\rangle=-2mf\langle A^2Y,AX\rangle+2m(\Hess f)(AX,Y),
$$
$$
2\sum\langle\nabla d\phi(X_i,Y),(\nabla^2\tau(\phi))(X_i,X)\rangle=-2mf\langle A^2X,AY\rangle+2m(\Hess f)(AY,X),
$$
$$
\sum\langle d\phi(X),R(X_i,Y)\nabla_{X_i}\tau(\phi)\rangle=-cmf\langle AX,Y\rangle+cm^2f^2\langle X,Y\rangle,
$$
$$
\sum\langle d\phi(Y),R(X_i,X)\nabla_{X_i}\tau(\phi)\rangle=-cmf\langle AX,Y\rangle+cm^2f^2\langle X,Y\rangle,
$$
$$
\langle d\phi(X), (\nabla R)(X_i , X_i ,Y, \tau(\phi)) + R(X_i ,Y)\nabla_{X_i}\tau(\phi)\rangle=0,
$$
$$
\langle d\phi(Y), (\nabla R)(X_i , X_i ,X, \tau(\phi)) + R(X_i ,X)\nabla_{X_i}\tau(\phi)\rangle=0,
$$
$$
\langle d\phi(X),\nabla_{\ricci(Y)}\tau(\phi)\rangle=-mf\langle AX,\ricci(Y)\rangle,
$$
$$
\langle d\phi(Y),\nabla_{\ricci(X)}\tau(\phi)\rangle=-mf\langle AY,\ricci(X)\rangle.
$$

Finally, for the remaining terms, we have
\begin{align*}
-\langle d\phi(X),\nabla_Y(\Delta\tau(\phi))\rangle&=-m\langle d\phi(X),\nabla_Y(\Delta H)\rangle\\&=-mY(\langle d\phi(X),\Delta H\rangle+m\langle\nabla_Yd\phi(X),\Delta H\rangle\\&=Y(\langle\tau_2^{\top}(\phi),X\rangle)+m\langle B(X,Y),\Delta H\rangle-\langle \nabla_XY,\tau_2^{\top}(\phi)\rangle
\end{align*}
and
$$
-\langle d\phi(Y),\nabla_X(\Delta\tau(\phi))\rangle=X(\langle\tau_2^{\top}(\phi),Y\rangle)+m\langle B(X,Y),\Delta H\rangle-\langle \nabla_YX,\tau_2^{\top}(\phi)\rangle.
$$

Assembling all these terms and taking into account that
$$
\tau_2^{\top}(\phi)=-2mA(\grad f)-\frac{m^2}{2}\grad f^2,
$$
one obtains
\begin{align*}
(\Delta S_2)(X,Y)=&\left(2cm^2f^2-\frac{m^2}{2}\Delta f^2\right)\langle X,Y\rangle\\&+2m^2f^2\langle AX,AY\rangle+2m^2(Xf)(Yf)+2cm(m-2)f\langle AX,Y\rangle\\&-2mf\langle\ricci(X),AY\rangle-2mf\langle\ricci(Y),AX\rangle-4mf\langle A^2X,AY\rangle\\&+2m(\Hess f)(AX,Y)+2m(\Hess f)(AY,X)\\&+\langle\nabla_Y\tau_2^{\top}(\phi),X\rangle+\langle\nabla_X\tau_2^{\top}(\phi),Y\rangle+2m\langle B(X,Y),\Delta H\rangle.
\end{align*}

Now, using that, in the case of hypersurfaces, $S_2=-(m^2f^2/2)I+2mfA$ and also
\begin{align*}
\langle\Hess f,A^2\rangle=&\langle\Hess f,c(m-1)I+mfA-\ricci \rangle\\=&-c(m-1)\Delta f-\frac{f}{2}\Div\left(\tau_2^{\top}(\phi)\right)+\frac{m^2f}{4}\Delta f^2\\&-m^2f|\grad f|^2-\Div(\ricci(\grad f))+\frac{1}{2}\langle\grad s,\grad f\rangle
\end{align*}
and
$$
\langle H,\Delta H\rangle=\frac{1}{2}\Delta f^2+f^2|A|^2+|\grad f|^2,
$$
a long but straightforward computation leads to the conclusion.
\end{proof}

\begin{remark}\label{rem:T}
Let $M^m$ be a hypersurface in a space form $N^{m+1}(c)$ and consider the operator $T$ on $M$ given by 
$$
T(X)=-\trace(RA)(\cdot,X,\cdot),
$$ 
where 
$$
RA(X,Y,Z)=R(X,Y)AZ-A(R(X,Y)Z),\quad\forall X,Y,Z\in C(TM).
$$
At a point $p\in M$, consider an orthonormal basis $\{e_i\}$ of $T_pM$ such that $Ae_i=\lambda_ie_i$. 

Using the operator $T$ we can write (see \cite{JHC})
\begin{align}\label{eq:T}
4cm^4f^4-4m^3f^3(\trace A^3)-4m^2f^2|A|^2\left(cm-|A|^2\right)=&4m^2f^2\langle T,A\rangle\\\nonumber=&-2m^2f^2\sum(\lambda_i-\lambda_j)^2R_{ijij}.
\end{align}
\end{remark}

The next result, which is obtained by a straightforward computation, comes to further improve the above formula of the Laplacian of $|S_2|^2$.

\begin{lemma}\label{p:3}
Let $M^m$ be a hypersurface in a space form $N^{m+1}(c)$. Then 
\begin{equation}\label{eq:nabla1}
|\nabla S_2|^2=(m-8)m^4f^2|\grad f|^2+4m^2|\nabla A_H|^2
\end{equation}
and, furthermore,
\begin{align}\label{eq:nabla2}
|\nabla S_2|^2=&(m-8)m^4f^2|\grad f|^2+4m^2|A|^2|\grad f|^2+4m^2f^2|\nabla A|^2\\\nonumber&+2m^2\Div\left(|A|^2\grad f^2\right)+2m^2|A|^2\Delta f^2.
\end{align}
\end{lemma}

The following two results are partial answers to Conjecture \ref{conj2}. 

\begin{proposition}\label{p:4}
Let $\phi:M^m\rightarrow\mathbb{S}^{m+1}$ be a compact proper-biharmonic hypersurface in the Euclidean sphere. If the scalar curvature $s$ of $M$ is constant, and
$$
mf^2\leq f(\trace A^3), 
$$
then $M$ is either $\mathbb{S}^m(1/\sqrt{2})$ or the product $\mathbb{S}^{m_1}(1/\sqrt{2})\times\mathbb{S}^{m_2}(1/\sqrt{2})$, $m_1+m_2=m$, $m_1\neq m_2$.
\end{proposition}

\begin{proof} Since $M$ is a compact biharmonic hypersurface with constant scalar curvature, we have, using Proposition \ref{p:delta}, Equation \eqref{eq:nabla1}, and \cite{MO},
\begin{align}
\int_M\left\{m^2f^4-mf^3(\trace A^3)\right\}=\int_{M}|\nabla A_H|^2.
\end{align}
It follows that $\nabla A_H=0$, which, for non-minimal hypersurfaces, is equivalent to $\nabla A=0$. We conclude using \cite{CMO1,GYJ}, where all proper-biharmonic hypersurfaces satisfying $\nabla A=0$ were determined.
\end{proof}

It is easy to see that, for a CMC biharmonic hypersurface of the Euclidean sphere, $\riem^M\geq 0$ implies $mf^2\leq f(\trace A^3)$, and then one obtains the following corollary.

\begin{corollary}\label{cor3final}
Let $\phi:M^m\rightarrow\mathbb{S}^{m+1}$ be a compact proper-biharmonic hypersurface with constant scalar curvature and $\riem^M\geq 0$. Then $M$ is either $\mathbb{S}^m(1/\sqrt{2})$ or the product $\mathbb{S}^{m_1}(1/\sqrt{2})\times\mathbb{S}^{m_2}(1/\sqrt{2})$, $m_1+m_2=m$, $m_1\neq m_2$.
\end{corollary}

\begin{remark} 
Note that if $M^m$ is a constant-scalar-curvature compact proper-biharmonic hypersurface in $\mathbb{S}^{m+1}$, then we have the following constraint (see \cite{O})
$$
s\in\left(m(m-2),2m(m-1)\right].
$$
In \cite{CY}, compact hypersurfaces $M^m$ in $\mathbb{S}^{m+1}$ with $\riem^M\geq 0$ and constant scalar curvature $s\geq m(m-1)$ were classified. If we ask $M$ to satisfy the hypotheses in Corollary \ref{cor3final} it does not necessarily follows that $s\geq m(m-1)$, but only $s> m(m-2)$.
\end{remark}

\begin{remark}
In the non-compact case, it is proved that a constant-scalar-curvature proper-biharmonic hypersurface of the Euclidean sphere with less than seven distinct principal curvatures must be CMC (see \cite{FH}).
\end{remark}

From Proposition \ref{p:delta} and the second equation in Lemma \ref{p:3} we obtain a further formula for the Laplacian of $|S_2|^2$.

\begin{theorem}\label{th:1}
Let $\phi:M^m\rightarrow N^{m+1}(c)$ be a hypersurface in a space form. Then
\begin{align}\label{eq:delta}
\frac{1}{2}\Delta |S_2|^2=&4cm^4f^4-4m^3f^3(\trace A^3)-4m^2f^2|A|^2(cm-|A|^2)\\\nonumber
&-8m^4f^2|\grad f|^2-4m^2f^2|\nabla A|^2\\\nonumber &+4m^2f\langle\grad s,\grad f\rangle\\\nonumber&-8m^2\Div(f\ricci(\grad f))-2m^2\Div\left(|A|^2\grad f^2\right)\\\nonumber & +\frac{m^5}{8}\Delta f^4-4cm^2(m-1)\Delta f^2-10m^2f\langle\tau_2^{\top}(\phi),\grad f\rangle\\\nonumber &-4m^2f^2\Div\left(\tau_2^{\top}(\phi)\right)-2\left|\tau_2^{\top}(\phi)\right|^2+4mf\langle\nabla\tau_2^{\top}(\phi),A\rangle.
\end{align}
\end{theorem}

\begin{remark} From Theorem \ref{th:1} one also obtains, by a straightforward computation, a formula for the Laplacian of the squared norm of the shape operator $A$ for any hypersurface in a space form, a result that generalizes the well-known formula for CMC hypersurfaces in \cite{NS}.
\end{remark}

Theorem \ref{th:1} leads to the next two results.

\begin{theorem} 
Let $\phi:M^m\rightarrow N^{m+1}(c)$ be a constant-scalar-curvature biconservative hypersurface in a space form. Then
\begin{align*}
\frac{3m^2}{2}\Delta f^4=&\ 4f^2\big\{cm^2f^2-mf(\trace A^3)-|A|^2(cm-|A|^2)\\
&-2m^2|\grad f|^2-|\nabla A|^2\big\}.
\end{align*}
\end{theorem}

\begin{corollary}
Let $\phi:M^m\rightarrow \mathbb{S}^{m+1}$ be a biharmonic hypersurface with constant scalar curvature. Then the following system holds
\begin{align}\label{system}
\begin{cases}
\frac{3m^2}{2}\Delta f^4=&4f^2\big\{m^2f^2-mf(\trace A^3)-|A|^2(m-|A|^2)\\
&-2m^2|\grad f|^2-|\nabla A|^2\big\}\\
\quad\quad\Delta f=&f(m-|A|^2).
\end{cases}
\end{align}
\end{corollary}

\begin{remark} Since $\Delta f^4=f^3\Delta f-12|\grad f|^2$, a consequence of the last corollary is that a biharmonic hypersurface in the Euclidean sphere with constant scalar curvature satisfies
\begin{align*}
\begin{cases}
\frac{3m^2}{2}\Delta f^4=&4f^2\big\{m^2f^2-mf(\trace A^3)-|A|^2(m-|A|^2)\\&-2m^2|\grad f|^2-|\nabla A|^2\big\}\\
\quad\ \ \Delta f^4=&4f^4(m-|A|^2)-12f^2|\grad f|^2.
\end{cases}
\end{align*}
\end{remark}

The next rigidity result is an almost direct application of the Simons type formula \eqref{eq:delta}.

\begin{theorem}\label{thm:main}
Let $\phi:M^m\rightarrow N^{m+1}(c)$ be a compact biconservative hypersurface in a space form $N^{m+1}(c)$, with $c\in\{-1,0,1\}$. If $M$ is not minimal, has constant scalar curvature and $\riem^M\geq 0$, then $M$ is either 
\begin{enumerate}

\item $\mathbb{S}^m(r)$, $r>0$, if $c\in\{-1,0\}$, i.e., $N$ is either the hyperbolic space $\mathbb{H}^{m+1}$ or the Euclidean space $\mathbb{E}^{m+1}$; or

\item $\mathbb{S}^m(r)$, $r\in(0,1)$, or the product $\mathbb{S}^{m_1}(r_1)\times\mathbb{S}^{m_2}(r_2)$, where $r_1^2+r_2^2=1$, $m_1+m_2=m$, and $r_1\neq\sqrt{m_1/m}$, if $c=1$, i.e., $N$ is the Euclidean sphere $\mathbb{S}^{m+1}$.
\end{enumerate}
\end{theorem}

\begin{proof} Integrating Equation \eqref{eq:delta} over $M$ with $s$ constant, we have
\begin{align}\label{eq:th1}
\int_{M}\left\{4cm^4f^4-4m^3f^3(\trace A^3)-4m^2f^2|A|^2\left(cm-|A|^2\right)\right\}=&\int_{M}\big\{8m^4f^2|\grad f|^2\\\nonumber&+4m^2f^2|\nabla A|^2\big\}\geq 0.
\end{align}

Since $\riem^M\geq 0$, Equations \eqref{eq:T} and \eqref{eq:th1} forces $f^2|\grad f|^2=0$ and $\nabla A=0$, which implies $T=0$. Therefore, $M$ is a CMC hypersurface with $\nabla A=0$ and we conclude using \cite{L,R}.
\end{proof}

\section{A Bochner type formula and applications}

The results related to Conjecture \ref{conj2} obtained in the previous section rely heavily on the constant scalar curvature hypothesis. To avoid this condition we will first prove a proposition inspired by \cite{MSY}, and based on a non-linear Bochner type formula using a $4$-tensor defined on a Riemannian manifold $M$:
$$
Q(X,Y,Z,W)=\langle Y,Z\rangle\langle X,W\rangle-\langle X,Z\rangle\langle Y,W\rangle
$$ 
and the map 
$$
\sigma_{24}(X,Y,Z,W)=(X,W,Z,Y),
$$ 
given a symmetric $(1,1)$-tensor $S$, the $1$-form $\theta$ defined as the contraction $C((Q\circ\sigma_{24})\otimes g^{\ast},\nabla S\otimes S))$, where $g$ denotes the metric tensor on $M$ and $g^{\ast}$ is its dual.

The next formula cannot be considered to be of Simons type as we do not compute a Laplacian and the shape operator is not necessarily involved. The formula extends beyond Codazzi tensors as it involves the antisymmetric part of $\nabla S$.

\begin{proposition}\label{p:MSY} On a Riemannian manifold $M$ with curvature tensor $R$ we have
\begin{equation}\label{eq:MSY}
\Div\theta=\langle T,S\rangle+|\Div S|^2-|\nabla S|^2+\frac{1}{2}|W|^2,
\end{equation}
where $T(X)=-\trace(RS)(\cdot,X,\cdot)$ and $W(X,Y)=(\nabla_XS)Y-(\nabla_YS)X$.
\end{proposition}

\begin{proof} Since we work with tensor products, it seems easier to use local coordinates. This way one can write
$$
(Q\circ\sigma_{24})\otimes g^{\ast}=Q_{ilkj}g^{ab}dx^i\otimes dx^{j}\otimes dx^{k}\otimes dx^{l}\otimes\frac{\partial}{\partial x^a}\otimes\frac{\partial}{\partial x^b},
$$
$$
(\nabla S)\otimes S=(\nabla_{\alpha}S_{\beta}^{\sigma})S_{\gamma}^{\delta}dx^{\alpha}\otimes dx^{\beta}\otimes dx^{\gamma}\otimes\frac{\partial}{\partial x^{\sigma}}\otimes\frac{\partial}{\partial x^{\delta}},
$$
and
$$
(\nabla S)\otimes(\nabla S)=(\nabla_{\alpha}S_{\beta}^{\sigma})(\nabla_{\gamma}S_{\omega}^{\delta})dx^{\alpha}\otimes dx^{\beta}\otimes dx^{\gamma}\otimes dx^{\omega}\otimes\frac{\partial}{\partial x^{\sigma}}\otimes\frac{\partial}{\partial x^{\delta}}.
$$
Therefore, we have
$$
\theta_i=Q_i^{j\ kl}(\nabla_{j}S_{l}^{a})S_{k}^{b}g_{ab}
$$
and then
$$
\theta^i=Q^{jikl}(\nabla_{j}S_{l}^{a})S_{ak}.
$$
Using this expression and the commutation formula for $\nabla^2S$, a straightforward computation leads to
\begin{align*}
\Div\theta=&\ Q^{jikl}\left\{(\nabla_i\nabla_jS_l^a)S_{ak}+(\nabla_jS_l^a)(\nabla_iS_{ak})\right\}\\=&\ Q^{jikl}\left\{(\nabla_i\nabla_jS_l^a)S_{ak}+(\nabla_jS_l^a)(\nabla_kS_{ia})+(\nabla_jS_l^a)(\nabla_iS_{ka}-\nabla_kS_{ia})\right\}\\=&\ \frac{1}{2}Q^{jikl}\{(S_l^bR_{bij}^a-S_b^aR_{lij}^b)S_{ak}+2(\nabla_jS_l^a)(\nabla_kS_{ia})\\&+2(\nabla_jS_l^a)(\nabla_iS_{ka}-\nabla_kS_{ia})\}\\=&\ \frac{1}{2}Q^{jikl}\{(RS(\partial_i,\partial_j,\partial_l))^aS_{ak}+2(\nabla_jS_l^a)(\nabla_kS_{ia})\\&+2(\nabla_jS_l^a)((\nabla_iS_{ka}-\nabla_kS_{ia}))\}.
\end{align*}
Since $Q^{jikl}=g^{ik}g^{jl}-g^{jk}g^{il}$, we get that
$$
\Div\theta=\frac{1}{2}g^{ik}g^{jl}(RS(\partial_i,\partial_j,\partial_l))^aS_{ak}+\frac{1}{2}\langle T,S\rangle+|\Div S|^2-|\nabla S|^2+\frac{1}{2}|W|^2.
$$

Next, let us consider a point $p\in M$ and $\{e_1,\ldots,e_m\}$ a basis at $p$ such that $Se_i=\lambda_ie_i$. Then one obtains
\begin{align*}
g^{ik}g^{jl}(RS(\partial_i,\partial_j,\partial_l))^aS_{ak}=&(g^{jl}S_l^bR_{bij}^a-g^{jl}S_b^aR_{lij}^b)S_a^i\\=&\sum_{i,k}\langle R(e_k,e_i)Se_i-S(R(e_k,e_i)e_i),Se_k\rangle\\=&-\frac{1}{2}\sum_{i,k}(\lambda_i-\lambda_k)^2R(e_k,e_i,e_k,e_i)\\=&\langle T,S\rangle
\end{align*}
and replacing in the expression of $\Div\theta$ we conclude.
\end{proof}

When $M$ is a compact CMC hypersurface in a space form, taking $A$ instead of $S$ in Equation \eqref{eq:MSY}, one obtains a classic formula from \cite{NS}. If $M$ is a biconservative surface, taking $S$ to be $S_2$, we recover \cite[Theorem 6]{LO}. Still with $S$ equals $S_2$, but for biharmonic hypersurfaces in Euclidean spheres, we get the following result.

\begin{proposition}\label{p:MSY1} Let $\phi:M^m\rightarrow \mathbb{S}^{m+1}$ be a compact proper-biharmonic hypersurface with $\riem^M\geq 0$, such that
$$
f^2|\nabla A|^2-|A|^2|\grad f|^2+|A|^2(m-|A|^2)f^2\geq 0.
$$
Then $M$ is either $\mathbb{S}^m(1/\sqrt{2})$ or the product $\mathbb{S}^{m_1}(1/\sqrt{2})\times\mathbb{S}^{m_2}(1/\sqrt{2})$, $m_1+m_2=m$, $m_1\neq m_2$.
\end{proposition}

\begin{proof} Recall that the biharmonic stress-energy tensor $S_2$ of a hypersurface is given by
$$
S_2=-\frac{m^2f^2}{2}I+2mfA,
$$
and a straightforward computation leads to 
\begin{align*}
|W|^2=\sum_{i,j}|W(X_i,X_j)|^2=&\ 2m^5f^2|\grad f|^2+8m^2|A|^2|\grad f|^2-10m^4f^2|\grad f|^2\\&+8m^3f\langle\grad f, A\grad f\rangle-8m^2|A\grad f|^2,
\end{align*}
where $\{X_i\}$ is a geodesic frame around a point $p\in M$.

From this formula and Lemma \ref{p:3} it follows that
\begin{align*}
\frac{1}{2}|W|^2-|\nabla S_2|^2=&-4m^2f^2|\nabla A|^2-2m^2\langle\grad f^2,\grad |A|^2\rangle\\=&-4m^2f^2|\nabla A|^2-2m^2\Div(|A|^2\grad f^2)-2m^2|A|^2\Delta f^2\\=&-4m^2\left\{f^2|\nabla A|^2-|A|^2|\grad f|^2+|A|^2(m-|A|^2)f^2\right\}\\&-2m^2\Div(|A|^2\grad f^2).
\end{align*}

Next, by integrating \eqref{eq:MSY} on $M$, from the hypotheses, it easily follows that
$$
f^2|\nabla A|^2-|A|^2|\grad f|^2+|A|^2(m-|A|^2)f^2=0
$$
and
\begin{align}\label{eq:eigen}
\sum_{i,j}f^2(\lambda_i-\lambda_j)^2(1+\lambda_i\lambda_j)=0,
\end{align}
where $\lambda_i$ are the principal curvatures of $M$.

Now, from \eqref{eq:eigen} it follows that, on a connected component of $U=\{p\in M|f^2(p)>0\}$, there are at most two distinct principal curvatures, not necessarily constant, and then, since $M$ is biharmonic, we have that $\grad f=0$, and so $\Delta f=0$, on that component and therefore on $U$ (see \cite{BMO}). Let $q\in M$ be a point such that $f(q)=0$. From the normal part of the biharmonic equation \eqref{eq:Jiang}, it can be easily seen that $(\Delta f)(q)=0$, which means that $\Delta f=0$ on $M$. Therefore, $f$ is constant on $M$, i.e., $M$ is a CMC hypersurface with at most two distinct principal curvatures, which implies $|A|^2=m$ and $\nabla A=0$. This concludes the proof.
\end{proof}

In \cite{JHC} it is proved that, for a biharmonic hypersurface $M^m$ in $\mathbb{S}^{m+1}$, we have
\begin{align}\label{eq:JHC}
|\nabla A|^2\geq \frac{m^2(m+26)}{4(m-1)}|\grad f|^2.
\end{align}
Using this inequality, one obtains the following corollary of Proposition \ref{p:MSY1}.

\begin{corollary}
Let $\phi:M^m\rightarrow \mathbb{S}^{m+1}$ be a compact proper-biharmonic hypersurface with $\riem^M\geq 0$, such that
$$
\left(\frac{m^2(m+26)}{4(m-1)}f^2-|A|^2\right)|\grad f|^2+|A|^2(m-|A|^2)f^2\geq 0.
$$
Then $M$ is either $\mathbb{S}^m(1/\sqrt{2})$ or the product $\mathbb{S}^{m_1}(1/\sqrt{2})\times\mathbb{S}^{m_2}(1/\sqrt{2})$, $m_1+m_2=m$, $m_1\neq m_2$.
\end{corollary}

In this last part we will use Equation \eqref{eq:MSY} to study biconservative submanifolds with parallel normalized mean curvature vector field in space forms.

A non-minimal submanifold in a Riemannian manifold with the mean curvature vector field parallel in the normal bundle is called a PMC submanifold.

Let now $\phi:M^m\rightarrow N^n$ be a submanifold with mean curvature vector field $H$ such that $H\neq 0$ at any point in $M$. Henceforth we will denote by $h=|H|>0$ the mean curvature of $M$ and by $\eta_0=H/|H|$ a unit normal vector field with the same direction as $H$. If $\eta_0$ is parallel in the normal bundle, i.e., $\nabla^{\perp}\eta_0=0$, the submanifold $M$ is said to have parallel normalized mean curvature vector field and it is then called a PNMC submanifold. It is easy to see that a PNMC submanifold is PMC if and only if it also is CMC.

Now, let us denote $A_0=A_{\eta_0}$ the shape operator of $M$ in the direction $\eta_0$.  We have the following straightforward properties of $A_0$. 

\begin{lemma}\label{l:A0}
Let $\phi:M^m\rightarrow N^n(c)$ be a PNMC submanifold in a space form. Then, the following hold$:$
\begin{enumerate}

\item $A_0$ is symmetric$;$

\item $\nabla A_0$ is symmetric$;$

\item $\langle(\nabla A_0)(\cdot,\cdot),\cdot\rangle$ is totally symmetric$;$

\item $\trace A_0=mh$$;$

\item $\Div A_0=\trace(\nabla A_0)=m\grad h$$.$

\end{enumerate}
\end{lemma}

We will need the following lemma, that provides an inequality similar to \eqref{eq:JHC}, for the last main result.

\begin{lemma}\label{l:ineq}
Let $\phi:M^m\rightarrow N^n(c)$ be a PNMC biconservative submanifold. Then
\begin{equation}\label{eq:nablaA0}
|\nabla A_0|^2\geq\frac{m^2(m+26)}{4(m-1)}|\grad h|^2.
\end{equation}
\end{lemma}

\begin{proof} Since $M$ is biconservative, we have $\Div S_2=0$, which is equivalent to 
$$
\trace(\nabla A_H)=\frac{m}{4}\grad h^2.
$$
We can rewrite this relation as follows. Consider a geodesic frame $\{X_i\}$ around a point $p\in M$. Then, at $p$, one obtains
$$
\sum_i(\nabla(hA_0))(X_i,X_i)=\frac{m}{4}\grad h^2
$$
and then
$$
\sum_i\left((X_ih)A_0+h\nabla_{X_i}A_0\right)(X_i)=\frac{m}{4}\grad h^2,
$$
that is
$$
A_0\grad h+h\Div A_0=\frac{m}{4}\grad h^2.
$$
From the last property in Lemma \ref{l:A0}, it follows that
\begin{equation}\label{eq:biA0}
A_0\grad h=-\frac{m}{2}h\grad h.
\end{equation}

Next, consider a point $p_0\in M$. If $\grad h$ vanishes at $p_0$, the inequality \eqref{eq:nablaA0} obviously holds. Assume that $(\grad h)(p_0)\neq 0$ and then $\grad h$ does not vanish throughout an open neighborhood of $p_0$. On this neighborhood consider an orthonormal frame field $\{E_1=\grad h/|\grad h|,E_2,\ldots,E_m\}$. Then, from \eqref{eq:biA0}, we have
\begin{equation}\label{eq:final1}
A_0E_1=-\frac{m}{2}hE_1.
\end{equation} 

Now, using Equation \eqref{eq:final1} and the fact that $A_0$ is symmetric, one obtains
\begin{align}\label{eq:final11}
\langle(\nabla A_0)(E_1,E_1),E_1\rangle=&\langle\nabla_{E_1}A_0E_1-A_0(\nabla_{E_1}E_1),E_1\rangle\\\nonumber=&\left\langle-\frac{m}{2}\nabla_{E_1}(hE_1)-A_0(\nabla_{E_1}E_1),E_1\right\rangle\\\nonumber=&\left\langle-\frac{m}{2}|\grad h|E_1-\frac{m}{2}h\nabla_{E_1}E_1-A_0(\nabla_{E_1}E_1),E_1\right\rangle\\\nonumber=&-\frac{m}{2}|\grad h|
\end{align}
and then, from the last property in Lemma \ref{l:A0}, we have
\begin{align}\label{eq:final2}
\sum_{i=2}^{m}\left\langle(\nabla A_0)(E_i,E_i),E_1\right\rangle=&\sum_{i=1}^{m}\left\langle(\nabla A_0)(E_i,E_i),E_1\right\rangle-\langle(\nabla A_0)(E_1,E_1),E_1\rangle\\\nonumber =&\langle\Div A_0,E_1\rangle+\frac{m}{2}|\grad h|\\\nonumber=&\frac{3m}{2}|\grad h|.
\end{align}

Finally, using \eqref{eq:final11}, \eqref{eq:final2}, and the third property in Lemma \ref{l:A0}, it follows that
\begin{align*}
|\nabla A_0|^2=&\sum_{i,j=1}^m|(\nabla A_0)(E_i,E_j)|^2=\sum_{i,j,k=1}^m\langle(\nabla A_0)(E_i,E_j),E_k\rangle^2\\\geq&\langle(\nabla A_0)(E_1,E_1),E_1\rangle^2+\sum_{i=2}^m\langle(\nabla A_0)(E_1,E_i),E_i\rangle^2\\&+\sum_{i=2}^m\langle(\nabla A_0)(E_i,E_1),E_i\rangle^2+\sum_{i=2}^m\langle(\nabla A_0)(E_i,E_i),E_1\rangle^2\\=&\frac{m^2}{4}|\grad h|^2+3\sum_{i=2}^m\langle(\nabla A_0)(E_i,E_i),E_1\rangle^2\\\geq&\frac{m^2}{4}|\grad h|^2+\frac{3}{m-1}\left(\sum_{i=2}^m\langle(\nabla A_0)(E_i,E_i),E_1\rangle\right)^2\\=&\frac{m^2(m+26)}{4(m-1)}|\grad h|^2
\end{align*}
and we are finished.
\end{proof}

We are now ready to prove the main result of this section.

\begin{theorem}\label{thm:pmc}
Let $\phi:M^m\rightarrow N^n(c)$ be a compact PNMC biconservative submanifold in a space form with $\riem^M\geq 0$ and $m\leq 10$. Then $M$ is a PMC submanifold and $\nabla A_H=0$.
\end{theorem}

\begin{proof} First take $S=A_0$ in Proposition \ref{p:MSY} and, since $A_0$ is a Codazzi tensor, by integrating over $M$ and using Lemma \ref{l:A0}, one obtains
\begin{equation}\label{eq:pmc1}
\int_{M}\left\{-\langle T,A_0\rangle+|\nabla A_0|^2\right\}=m^2\int_{M}|\grad h|^2.
\end{equation}

Next, using Inequality \eqref{eq:nablaA0}, we can see that
\begin{equation}\label{eq:pmc2}
\int_{M}\langle T,A_0\rangle\geq\frac{3m^2(10-m)}{4(m-1)}\int_{M}|\grad h|^2.
\end{equation}

But $\langle T, A_0\rangle=-(1/2)\sum_{i,j}(\lambda_i-\lambda_j)^2R(e_i,e_j,e_i,e_j)\leq 0$ at any point $p\in M$, where $\{e_1,\ldots,e_m\}$ is a basis at $p$ such that $A_0e_i=\lambda_ie_i$, and then, from \eqref{eq:pmc2}, it follows that, if $m\leq 9$, then $\grad h=0$, i.e., $h$ is constant and $\langle T, A_0\rangle=0$. Again using \eqref{eq:pmc1} we have that $\nabla A_0=0$ and therefore $\nabla A_H=0$.

When $m=10$, we can see from \eqref{eq:pmc2} that $\langle T,A_0\rangle=0$ and then, from \eqref{eq:pmc1}, that
$$
\int_M|\nabla A_0|^2=100\int_M|\grad h|^2,
$$
which shows that we must have equality in \eqref{eq:nablaA0}.

Consider now the open set $U=\{p\in M|(\grad h)(p)\neq 0\}$ and an arbitrary point $p_0\in U$. We will show that $\Delta h^2=0$ at $p_0$, and therefore on $U$. 

First, on an open neighborhood of $p_0$, we consider an othonormal frame field $\{E_1=\grad h/|\grad h|,E_2,\ldots,E_{10}\}$ and, since $A_0E_1=-5hE_1$, we have
\begin{align}\begin{cases}\label{eq:formulas}
(\nabla A_0)(E_1,E_1)=-5|\grad h|E_1\\ (\nabla A_0)(E_i,E_j)=0,\quad\forall i,j\in\{2,\ldots,10\},\quad i\neq j\\
(\nabla A_0)(E_1,E_i)=\frac{5}{3}|\grad h|E_i,\quad\forall i\in\{2,\ldots,10\}\\
(\nabla A_0)(E_i,E_i)=\frac{5}{3}|\grad h|E_1,\quad\forall i\in\{2,\ldots,10\}.
\end{cases}
\end{align}

From the commutation formula
$$
(\nabla^2A_0)(X,Y,Z)-(\nabla^2A_0)(Y,X,Z)=RA_0(X,Y,Z),
$$
one obtains
$$
\sum_{i=1}^{10}\left\{(\nabla^2A_0)(E_i,Y,E_i)-(\nabla^2A_0)(Y,E_i,E_i)\right\}=-T(Y).
$$
Then, since $\langle T,A_0\rangle=0$, we have
\begin{equation}\label{eq:com}
\sum_{i,j=1}^{10}\left\{\langle(\nabla^2A_0)(E_i,E_j,E_i),A_0E_j\rangle-\langle(\nabla^2A_0)(E_j,E_i,E_i),A_0E_j\rangle\right\}=0.
\end{equation}

After some long but otherwise simple computations, using Equations \eqref{eq:formulas} and $A_0E_1=-5hE_1$, we get the expressions of $(\nabla^2A_0)(E_1,E_1,E_1)$, $(\nabla^2A_0)(E_i,E_1,E_i)$, $(\nabla^2A_0)(E_1,E_j,E_1)$, $(\nabla^2A_0)(E_j,E_j,E_j)$, and $(\nabla^2A_0)(E_i,E_j,E_i)$, with $i,j\neq 1$ and $i\neq j$, and then
\begin{align*}
\sum_{i,j=1}^{10}\langle(\nabla^2A_0)(E_i,E_j,E_i),A_0E_j\rangle=&50h(E_1|\grad h|)+\frac{200}{3}h(\Div E_1)|\grad h|\\&+\frac{10}{3}|\grad h|\sum_{i=2}^{10}\langle\nabla_{E_i}E_1,A_0E_i\rangle
\end{align*}
and
$$
\sum_{i,j=1}^{10}\langle(\nabla^2A_0)(E_j,E_i,E_i),A_0E_j\rangle=-50h(E_1|\grad h|)+10|\grad h|\sum_{i=2}^{10}\langle\nabla_{E_i}E_1,A_0E_i\rangle.
$$

Replacing in Equation \eqref{eq:com}, one obtains
\begin{equation}\label{eq:com2}
15h(E_1|\grad h|)+10h(\Div E_1)|\grad h|-|\grad h|\sum_{i=2}^{10}\langle\nabla_{E_i}E_1,A_0E_i\rangle=0.
\end{equation}

We also have
\begin{align*}
\sum_{i=2}^{10}\langle\nabla_{E_i}E_1,A_0E_i\rangle=&-\sum_{i=2}^{10}\langle E_1,(\nabla A_0)(E_i,E_i)+A_0(\nabla_{E_i}E_i)\rangle\\=&-\langle E_1,15\grad h\rangle-\sum_{i=2}^{10}\langle A_0E_1,\nabla_{E_i}E_i\rangle\\=&-15|\grad h|-5h\sum_{i=2}^{10}\langle\nabla_{E_i}E_1,E_i\rangle\\=&-15|\grad h|-5h\Div E_1
\end{align*}
and Equation \eqref{eq:com2} becomes
\begin{equation}\label{eq:com3}
h(E_1|\grad h|)+h(\Div E_1)|\grad h|+|\grad h|^2=0.
\end{equation}

Now, we obtain $E_1|\grad h|=(\Hess h)(E_1,E_1)$ and 
$$
\Div E_1=-\frac{(\Hess h)(E_1,E_1)+\Delta h}{|\grad h|}
$$
and then, from \eqref{eq:com3}, it follows that
$$
-h\Delta h+|\grad h|^2=0,
$$
which is nothing but $\Delta h^2=0$. 

Next, on the $\Int(M\setminus U)$ we have $\grad h=0$ and therefore $\Delta h^2=0$. By continuity, it follows that $\Delta h^2=0$ throughout $M$, which means that $h$ is constant, i.e., $M$ is PMC. This also implies that $\nabla A_0=0$ and, therefore, that $\nabla A_H=0$, which concludes the proof.
\end{proof}

\begin{remark} 
The (compact) PMC submanifolds in $N(c)$, $c\in\{0,1\}$, with $A_H$ parallel were classified in \cite{BS,Y}.
\end{remark}

\begin{corollary}
Let $\phi:M^m\rightarrow N^{m+1}(c)$ be a compact biconservative hypersurface in a space form such that its mean curvature does not vanish at any point, $\riem^M\geq 0$, and $m\leq 10$. Then $M$ is one of the hypersurfaces in Theorem \ref{thm:main}.
\end{corollary}

From the last corollary we find another partial answer to Conjecture \ref{conj2}, which is a weaker result than that of J.-H. Chen \cite{JHC}.

\begin{corollary}
Let $\phi:M^m\rightarrow \mathbb{S}^{m+1}$ be a compact proper-biharmonic hypersurface such that its mean curvature does not vanish at any point, $\riem^M\geq 0$, and $m\leq 10$. Then $M$ is either $\mathbb{S}^m(1/\sqrt{2})$ or the product $\mathbb{S}^{m_1}(1/\sqrt{2})\times\mathbb{S}^{m_2}(1/\sqrt{2})$, $m_1+m_2=m$, $m_1\neq m_2$.
\end{corollary}

\section*{Open problems}

Our results concerning compact biconservative hypersurfaces in space forms satisfying certain additional geometric conditions raise the following natural question.

\vspace{0.5cm}

\textit{Is any compact biconservative hypersurface in a space form CMC?}

\vspace{0.5cm}

Another open problem is the following (possible) partial answer to Conjecture \ref{conj2}

\vspace{0.5cm}

\textit{The only non-minimal solutions to Equations \eqref{system} are the hypersurfaces given by Conjecture \ref{conj2}.}

\end{document}